\documentclass[12pt,reqno]{amsart}
\usepackage{amssymb, a4wide}
\usepackage[all]{xy}
\usepackage{amsmath,amsthm,amscd}
\usepackage{enumitem}
\usepackage{mathtools}
\usepackage{graphicx}

\newtheorem{theorem}{Theorem}[section]
\newtheorem{proposition}[theorem]{Proposition}
\newtheorem{lemma}[theorem]{Lemma}

\theoremstyle{definition}
\newtheorem{definition}[theorem]{Definition}

\theoremstyle{remark}
\newtheorem{remark}[theorem]{Remark}

\numberwithin{equation}{section}

\def\t{\otimes}

\def\al{\alpha}
\def\be{\beta}

\def\vp{\varphi}

\def\lra{\longrightarrow}
\def\rra{\rightarrow}

\def\ol{\overline}

\def\g{\operatorname{\mathfrak{g}}}
\def\h{\operatorname{\mathfrak{h}}}
\def\p{\operatorname{\mathfrak{p}}}
\def\q{\operatorname{\mathfrak{q}}}

\DeclareMathOperator{\U}{\mathtt{U}}
\DeclareMathOperator{\I}{\mathtt{I}}
\DeclareMathOperator{\J}{\mathtt{J}}

\DeclareMathOperator{\UL}{\mathtt{UL}}

\def\Ker{\operatorname{Ker}}

\def\Hom{\operatorname{Hom}}

\def\As{\operatorname{\textbf{\textsf{Alg}}}}
\def\XAs{\operatorname{\textbf{\textsf{XAlg}}}}
\def\XLie{\operatorname{\textbf{\textsf{XLie}}}}
\def\Lie{\operatorname{\textbf{\textsf{Lie}}}}
\def\Lb{\operatorname{\textbf{\textsf{Lb}}}}
\def\XLb{\operatorname{\textbf{\textsf{XLb}}}}

\def\cat{\operatorname{cat}}

\def\id{\operatorname{id}}

\newcommand{\lb}[1]{\mathfrak{#1}}
\def\T{\operatorname{T}}
\def\End{\operatorname{End}}
\DeclareMathOperator{\XUL}{\mathtt{XUL}}

\newcommand{\LM}{\mathcal{LM}}
\DeclareMathOperator{\Liez}{\texttt{Lie}}
\DeclareMathOperator{\XLiez}{\texttt{XLie}}

\begin{document}

\title[A natural extension of the universal enveloping algebra]{A natural extension of the universal enveloping algebra functor to crossed modules of Leibniz algebras}

\author{Rafael ~F.~Casado}
\address{Department of Algebra, Universidade de Santiago de Compostela\\
15782 Santiago de Compostela, Spain}
\email{rapha.fdez@gmail.com}
\author{Xabier Garc\'ia-Mart\'inez}
\address{Department of Algebra, Universidade de Santiago de Compostela\\
15782 Santiago de Compostela, Spain}
\email{xabier.garcia@usc.es}
\author{Manuel Ladra}
\address{Department of Algebra, Universidade de Santiago de Compostela\\
15782 Santiago de Compostela, Spain}
\email{manuel.ladra@usc.es}

\thanks{The authors were supported by Ministerio de Econom\'ia y Competitividad (Spain), grant MTM2013-43687-P (European FEDER support included).
 The second and third authors were supported by Xunta de Galicia, grant GRC2013-045 (European FEDER support included).
  The second author is also supported by an FPU scholarship, Ministerio de Educaci\'on, Cultura y Deporte (Spain).}

\begin{abstract}
 The universal enveloping algebra functor between Leibniz and associative algebras defined by Loday and Pirashvili is extended to crossed modules.
  We prove that the universal enveloping crossed module of algebras of a crossed module of Leibniz algebras is its natural generalization.
   Then we construct an isomorphism between the category of representations of a Leibniz crossed module and the category of left modules over its universal enveloping crossed module of algebras.
    Our approach is particularly interesting since the actor in the category of Leibniz crossed modules does not exist in general,
     so the technique used in the proof for the Lie case cannot be applied.
      Finally we move on to the framework of the Loday-Pirashvili category $\LM$ in order to comprehend this universal enveloping crossed module in terms of the Lie crossed modules case.
\end{abstract}
\maketitle

\section{Introduction}
Leibniz algebras, which are a non-antisymmetric generalization of Lie algebras, were introduced in 1965 by Bloh \cite{Bloh},
who called them $D$-algebras and referred to the well-known \emph{Leibniz identity} as \emph{differential identity}.
 In 1993 Loday \cite{Lo_Lbalg} made them popular and studied their (co)homology. From that moment, many authors have studied this structure,
  obtaining very relevant algebraic results \cite{LoPi,LoPi2} and applications to Geometry \cite{KiWe,Lodd} and Physics \cite{FeLoOn}.

Crossed modules of groups were described for the first time by Whitehead in the late 1940s \cite{Wh} as an algebraic model for path-connected CW-spaces
 whose homotopy groups are trivial in dimensions greater than $2$. From that moment, crossed modules of different algebraic objects, not only groups,
  have been considered, either as tools or as algebraic structures in their own right. For instance, in \cite{LoPi} crossed modules of Leibniz algebras were defined in order to study cohomology.

Observe that in Ellis's PhD thesis \cite{Ellis} it is proved that, given a category of $\Omega$-groups $\mathcal{C}$ such as the categories of associative and Leibniz algebras,
 crossed modules, $\cat^1$-objects, internal categories and simplicial objects in $\mathcal{C}$ whose Moore complexes are of length 1 are equivalent structures.
  Internal categories can be described in terms of what Baez calls strict $2$-dimensional objects (see \cite{BaLa} for groups and \cite{BaCr} for Lie algebras).
    By analogy to Baez's terminology, crossed modules of associative algebras (respectively Leibniz algebras)
     can be viewed as strict associative $2$-algebras \cite{Kh} (respectively strict Leibniz $2$-algebras \cite{FCas,ShLi}).

In the case of Lie algebras, the universal enveloping algebra plays two important roles: the category of representations of a Lie algebra is isomorphic
 to the category of left modules over its universal enveloping algebra and it is right adjoint to the Liezation functor. For Leibniz algebras,
  these roles are played by two different functors: Loday and Pirashvili \cite{LoPi} proved that, given a Leibniz algebra $\lb{p}$,
   the category of left $\UL(\lb{p})$-modules is isomorphic to the category of $\lb{p}$-representations, where $\UL(\lb{p})$
   is the universal enveloping associative algebra of $\lb{p}$. On the other hand, if associative algebras are replaced by dialgebras,
    there exist a universal enveloping dialgebra functor \cite{Lo3}, which is right adjoint to the functor that assigns to every dialgebra its corresponding Leibniz algebra.

Another very interesting point of view on the construction of $\UL(\p)$ is explored in \cite{LoPi2}. They introduce the tensor category of linear maps $\LM$,
 also known as the Loday-Pirashvili category. It is possible to define Lie and associative objects in that category and to construct the universal enveloping algebra.
  Since a Leibniz algebra can be considered as a Lie object in $\LM$, it is remarkable that $\UL(\p)$ can be obtained via the universal enveloping algebra
   of the aforementioned Lie object in $\LM$. This approach is especially useful for studying the universal enveloping algebra of a Leibniz algebra in terms of Lie algebras.

As Norrie states in \cite{No}, it is surprising the ease of the generalization to crossed modules of many properties satisfied by the objects in the base category.
 In \cite{CaCaKhLa2}, the universal enveloping dialgebra functor is extended to crossed modules. The aim of this article is to extend to crossed modules the functor $\UL$,
  to prove that the aforementioned isomorphism between representations of a Leibniz algebra and left modules over its universal enveloping algebra also exists,
   and to study the $2$-dimensional version of $\UL$ in terms of Lie objects in $\LM$.

Observe that the analogous isomorphism between representations and left modules in the case of Lie crossed modules can be easily proved via the actor,
 but this method cannot be applied in our case, since the actor of a Leibniz crossed module does not necessarily exist \cite{FCas} (see \cite{CaDaLa} for the $1$-dimensional case).
  This makes our approach especially interesting.

In Section~\ref{section_prelim} we recall some basic definitions and properties, such as the concept of crossed module of associative and Leibniz algebras,
 along with the notions of the corresponding $\cat^1$-objects.
In Section~\ref{section_represen} we give proper definitions of left modules over a crossed module of associative algebras and representations of a Leibniz crossed module.
In Section~\ref{section_extension} we describe the generalization to crossed modules of the functor $\UL \colon \Lb \to \As$, that is $\XUL \colon \XLb \to \XAs$,
 which assigns to every Leibniz crossed module its corresponding universal enveloping crossed module of algebras. Additionally,
  we prove that $\XUL$ is a natural generalization of $\UL$, in the sense that it commutes (or commutes up to isomorphism)
   with the two reasonable ways of regarding associative and Leibniz algebras as crossed modules.
In Section~\ref{section_isomorph} we construct an isomorphism between the categories of representations of a Leibniz crossed module
and the left modules over its universal enveloping crossed module of algebras.
Finally, in Section~\ref{section_LM}, we introduce Lie and associative crossed modules in the Loday-Pirashvili category and
we prove that the factorization of the crossed module $\XUL(\p)$ via Lie crossed modules in $\LM$ also holds in the $2$-dimensional case.

\subsection*{Notations and conventions} Throughout the paper, we fix a commutative ring $K$ with unit. All algebras are considered over $K$.
 The categories of Lie, Leibniz and (non-unital) associative algebras will be denoted by $\Lie$, $\Lb$ and $\As$, respectively.
\section{Preliminaries}\label{section_prelim}
\begin{definition}[\cite{Lo_Lbalg}]
	A \emph{Leibniz algebra $\lb{p}$ over $K$} is a $K$-module together with a bilinear operation $[ \ , \ ]\colon\lb{p}\times\lb{p}\to\lb{p}$,
 called the Leibniz bracket, which satisfies the Leibniz identity:
	\[
	[[p_1,p_2],p_3]=[p_1,[p_2,p_3]]+[[p_1,p_3],p_2],
	\]
 for all $p_1,p_2,p_3\in\lb{p}$. A morphism of Leibniz algebras is a $K$-linear map that preserves the bracket.
\end{definition}
We will denote by $\Lb$ the category of Leibniz algebras and morphisms of Leibniz algebras. These are in fact right Leibniz algebras.
 For the opposite structure, that is $[p_1,p_2]'=[p_2,p_1]$, the left Leibniz identity is
\[
[p_1,[p_2,p_3]']'=[[p_1,p_2]',p_3]'+[p_2,[p_1,p_3]']',
\]
 for all $p_1,p_2,p_3\in\lb{p}$.

If the bracket of a Leibniz algebra $\lb{p}$ happens to be antisymmetric, then $\lb{p}$ is a Lie algebra. Furthermore, every Lie algebra is a Leibniz algebra. For more examples, see \cite{Lo_Lbalg}.

Recall that a Leibniz algebra $\lb{p}$ acts on another Leibniz algebra $\lb{q}$ if there are two bilinear maps
$\lb{p}\times \lb{q}\to \lb{q}$, $(p,q)\mapsto \left[p,q\right]$ and $\lb{q}\times \lb{p}\to \lb{q}$, $(q,p)\mapsto \left[q,p\right]$, satisfying six relations,
 which are obtained from the Leibniz identity by taking two elements in $\lb{p}$ and one in $\lb{q}$ (three identities) and one element in $\lb{p}$
 and two elements in $\lb{q}$ (three identities). Given an action of a Leibniz algebra $\lb{p}$ on another Leibniz algebra $\lb{q}$,
  it is possible to consider the semidirect product $\lb{q} \rtimes \lb{p}$, whose Leibniz structure is given by:
\[
[(q_1,p_1),(q_2,p_2)]=([q_1,q_2]+[p_1,q_2]+[q_1,p_2], [p_1,p_2]),
\]
for all $(q_1,p_1), \, (q_2,p_2)\in \lb{q} \oplus \lb{p}$.

\begin{definition}[\cite{LoPi}]
	A \emph{representation of a Leibniz algebra $\lb{p}$} is a $K$-module $M$ equipped with two actions
 $\lb{p}\times M\to M$, $(p,m)\mapsto [p,m]$ and $M\times \lb{p}\to M$, $(m,p)\mapsto [m,p]$, satisfying the following three axioms:
	\begin{align*}
		[m,[p_1,p_2]] & = [[m,p_1],p_2]-[[m,p_2],p_1],\\
		[p_1,[m,p_2]] & = [[p_1,m],p_2]-[[p_1,p_2],m],\\
		[p_1,[p_2,m]] & = [[p_1,p_2],m]-[[p_1,m],p_2],
	\end{align*}
for all $m \in M$ and $p_1, p_2\in \lb{p}$.
\end{definition}
A morphism $f\colon M \to N$  of $\lb{p}$-representations is a $K$-linear map which is
compatible with the left and right actions of $\lb{p}$.
\begin{remark}\label{defrep}
	Given a $\lb{p}$-representation $M$, we can endow the direct sum of $K$-modules $M \oplus \lb{p}$ with a Leibniz structure such that $M$ is an abelian ideal and $\lb{p}$ is a subalgebra.
 The converse statement is also true. It is evident that the Leibniz structure of $M \oplus \lb{p}$ is the one of $M \rtimes \lb{p}$, as described previously.
\end{remark}
\begin{definition}[\cite{LoPi}]\label{def_uea}
	Let $\lb{p}^{l}$  and $\lb{p}^r$ be two copies of a Leibniz algebra $\lb{p}$. We will denote by $x_l$ and $x_r$ the elements of $\lb{p}^{l}$ and $\lb{p}^r$ corresponding
	to $x\in \lb{p}$. Consider the tensor $K$-algebra $\T(\lb{p}^l \oplus \lb{p}^r)$, which is associative and unital. Let $I$ be the two-sided ideal corresponding to the relations:
	\begin{align*}
		& [x,y]_r = x_r y_r - y_r x_r,\\	
		& [x,y]_l = x_l y_r - y_r x_l,\\
		& (y_r+y_l)x_l = 0.	
	\end{align*}
 for all $x,y \in \lb{p}$. The \emph{universal enveloping algebra of the Leibniz algebra $\lb{p}$} is the
	associative and unital algebra
	\[
		\UL(\lb{p}) \coloneqq \T(\lb{p}^l \oplus \lb{p}^r) / I.
	\]
This construction defines a functor $\UL \colon \Lb \to \As$.
\end{definition}

\begin{theorem}[\cite{LoPi}]\label{theo_equiv_LoPi}
	The category of representations of the Leibniz algebra $\lb{p}$ is isomorphic to the category of left modules over  $\UL(\lb{p})$.
\end{theorem}
\begin{proof}
	Let $M$ be a representation of $\lb{p}$. It is possible to define a left action of $\UL(\lb{p})$ on the $K$-module $M$ as follows. Given $x_l \in \lb{p}^l$, $x_r \in \lb{p}^r$ and $m \in M$,
	\[
		x_l \cdot m = [x,m], \qquad x_r \cdot m = [m,x].
	\]
	These actions can be extended to an action of $\T(\lb{p}^l \oplus \lb{p}^r)$ by composition and linearity. It is not complicated to check that this way $M$ is equipped with a structure of left $\UL(\lb{p})$-module.
	
	Regarding the converse statement, it is immediate that, starting with a left $\UL(\lb{p})$-module, the
	restrictions of the actions to $\lb{p}^l$ and $\lb{p}^r$ give two actions of $\lb{p}$ which make $M$ into a
	representation.
\end{proof}
Recall that a left module over an associative algebra $A$ can be described as a morphism $\al \colon A \to \End(M)$, where $M$ is a $K$-module.

Both $\Lb$ and $\As$ are categories of interest, notion introduced by Orzech in \cite{Or}. See \cite{Mo} for a proper definition and more examples.
 Categories of interest are a particular case of categories of groups with operations, for which Porter \cite{Po} described the notion of crossed module.
  The following definitions agree with the one given by Porter.
\begin{definition}
	A \emph{crossed module of Leibniz algebras} (or \emph{Leibniz crossed module}) $(\lb{q}, \lb{p}, \eta)$ is a
	morphism of Leibniz algebras $\eta \colon \lb{q}\to \lb{p}$ together with an action of $\lb{p}$ on $\lb{q}$ such that
	\begin{align*}
	& \eta([p,q])=[p,\eta(q)] \qquad \text{and} \qquad \eta([q,p])=[\eta(q),p],\\
	& [\eta(q_1),q_2]=[q_1,q_2]=[q_1,\eta (q_2)],
	\end{align*}
	for all $q,q_1,q_2 \in \lb{q}$, $p \in \lb{p}$.
\end{definition}
	A \emph{morphism of Leibniz crossed modules} $(\varphi, \psi)$ from $(\lb{q},\lb{p},\eta)$ to $(\lb{q'},\lb{p'}, \eta')$
	is a pair of Leibniz homomorphisms, $\varphi \colon \lb{q} \to \lb{q}'$ and $\psi \colon \lb{p} \to \lb{p}'$, such that
	\begin{align*}
	& \psi \eta = \eta' \varphi, \\
	& \varphi([p,q]) = [\psi(p),\varphi(q)] \qquad \text{and} \qquad \varphi([q,p]) = [\varphi(q),\psi(p)],
	\end{align*}
	 for all $q\in \lb{q}$, $p\in \lb{p}$.
	
\begin{definition}
	A \emph{crossed module of algebras $(B, A,\rho)$} is an algebra
	homomorphism $\rho \colon B\rra A$ together with an action of $A$ on
	$B$ such that
	\begin{align*}
	& \rho(a b) = a \rho(b) \qquad \text{and} \qquad \rho(b a)=\rho(b) a,\\
	& \rho(b_1) b_2  = b_1 b_2 = b_1 \rho(b_2),
	\end{align*}
	for all $a\in A$, $b_1, b_2\in B$.
\end{definition}
A \emph{morphism of crossed modules of algebras} $(\varphi,\psi) \colon (B,A,\rho)\to(B',A',\rho')$ is a pair of algebra homomorphisms, $\varphi \colon B \to B'$ and $\psi \colon A \to A'$, such that
\begin{align*}
&\psi \rho = \rho' \varphi, \\
&\varphi(b a)  = \varphi(b) \psi(a) \qquad \text{and} \qquad \varphi(a b) = \psi(a) \varphi(b),
\end{align*}
for all $b\in B$, $a\in A$.

We will denote by $\XLie$, $\XLb$ and $\XAs$ the categories of Lie crossed modules,  Leibniz crossed modules and crossed modules of associative algebras, respectively.
 Crossed modules can be alternatively describe as $\cat^1$-objects, namely:
\begin{definition}
	A \emph{$\cat^1$-Leibniz algebra} $(\lb{p}_1,\lb{p}_0,s,t)$ consists of a Leibniz algebra $\lb{p}_1$
	together with a Leibniz subalgebra $\lb{p}_0$ and the structural morphisms $s,t \colon \lb{p}_1\to \lb{p}_0$  such that
	\begin{align}
	& s|_{\lb{p}_0} = t|_{\lb{p}_0} = \id_{\lb{p}_0}, \label{CLb1} \tag{CLb1} \\
	& [\Ker s, \Ker t]= 0 = [\Ker t, \Ker s],\label{CLb2} \tag{CLb2}
	\end{align}
\end{definition}
\begin{definition}
	A \emph{$\cat^1$-algebra} $(A_1,A_0,\sigma,\tau)$ consists of an algebra $A_1$
	together with a subalgebra $A_0$ and the structural morphisms $\sigma,\tau \colon A_1\to A_0$  such that
	\begin{align}
	& \sigma|_{A_0} = \tau|_{A_0} = \id_{A_0}\label{CAs1}\tag{CAs1}, \\
	& \Ker \sigma \Ker \tau= 0= \Ker \tau \Ker \sigma.\label{CAs2}\tag{CAs2}
	\end{align}
\end{definition}
It is a well-known fact  that the category of crossed modules of Leibniz algebras (resp. associative algebras) is equivalent to the category of $\cat^1$-Leibniz algebras (resp. associative algebras)
 (see for instance \cite{Ellis} in the framework of  categories of $\Omega$-groups).

Given a crossed module of Leibniz algebras $(\lb{q},\mathfrak{p},\eta)$, the corresponding $\cat^1$-Leibniz algebra is $(\lb{q} \rtimes \lb{p}, \lb{p}, s, t)$,
	where $s(q, p) = p$ and $t(q, p) = \eta(q)+p$ for all $(q, p) \in \lb{q} \rtimes \lb{p}$. Conversely, given a $\cat^1$-Leibniz algebra  $(\lb{p}_1,\lb{p}_0,s,t)$,
the corresponding Leibniz crossed module is $t|_{\Ker s} \colon \Ker s\to \lb{p}_0$, with the action of $\lb{p}_0$ on $\Ker s$ induced by the bracket in $\lb{p}_1$.
 The equivalence for associative algebras is analogous.

The standard functor liezation, $\Liez \colon \Lb \to \Lie$,  $\p \mapsto \Liez(\p)$, where $\Liez(\p)$ is the quotient of $\p$ by the ideal generated by the elements $[p,p]$, for $p \in \p$,
can be extended to crossed modules  $\XLiez \colon \XLb \to \XLie$.

 Given a Leibniz crossed module $(\lb{q},\lb{p},\eta)$ its liezation $\XLiez(\lb{q},\lb{p},\eta)$  is defined as the crossed module $\big(\dfrac{\Liez(\q)}{[\q,\p]_\texttt{x}}, \Liez(\p), \overline{\eta}\big)$,
 where  $[\q,\p]_\texttt{x}$ is the ideal  generated by the elements $[q, p] + [p, q]$,  for $p \in \p, q \in \q$.

\section{Representations of crossed modules}\label{section_represen}
Since our intention is to extend Theorem~\ref{theo_equiv_LoPi} to crossed modules,
 it is necessary to give a proper definition of representations over Leibniz crossed modules and left modules over crossed modules of algebras.

In the case of crossed modules of associative algebras, by analogy to the $1$-dimensional situation, left modules can be described via the endomorphism crossed module:
\begin{definition}
	Let $(B,A,\rho)$ be a crossed module of algebras. A \emph{left $(B,A,\rho)$-module} is an abelian crossed module of algebras $(V,W,\delta)$,
 that is $\delta$ is simply a morphism of $K$-modules and the action of $W$ on $V$ is trivial,
	together with a morphism of crossed modules of algebras $(\vp,\psi) \colon (B,A,\rho) \to (\Hom_K(W,V), \End(V,W,\delta), \Gamma)$.
\end{definition}
Note that $\End(V,W,\delta)$ is the algebra of all pairs $(\al,\be)$, with $\al\in \End_K(V)$
and $\be\in \End_K(W)$, such that $\be\delta=\delta\al$. Furthermore, $\Gamma(d)=(d \delta, \delta d)$ for all $d \in \Hom_K(W,V)$. The action of $\End(V,W,\delta)$ on $\Hom_K(W,V)$ is given by
\[
{(\al,\be)}\cdot d= \al d \qquad \text{and} \qquad d\cdot (\al,\be)=d \be,
\]
for all $d\in \Hom_K(W,V)$, $(\al,\be)\in \End(V,W,\delta)$. See \cite{CaCaKhLa,CaInKhLa} for further details.

Let $(V,W,\delta)$ and $(V',W',\delta')$ be left $(B,A,\rho)$-modules with the corresponding homomorphisms
of crossed modules of algebras $(\vp,\psi) \colon (B,A,\rho) \to (\Hom_{K}(W,V), \End(V,W,\delta), \Gamma)$
and $(\vp',\psi') \colon (B,A,\rho) \to (\Hom_{K}(W',V'), \End(V',W',\delta'), \Gamma')$. Then a morphism from $(V,W,\delta)$ to $(V',W',\delta')$ is a pair $(f_{V},f_{W})$ of morphisms of $K$-modules
${f}_{V} \colon V \to V'$ and ${f}_{W} \colon W \to W'$ such that
\begin{align*}
f_{W} \delta & = \delta' f_{V}, \\
(f_{V},f_{W}) \psi(a) & = \psi'(a) (f_{V},f_{W}),\\
f_{V} \vp(b) & = \vp'(b) f_{W},
\end{align*}
 for all $b \in B$, $a \in A$.

For the categories of crossed modules of groups and Lie algebras, representations can be defined via an object called the actor (see \cite{CaLa,No}).
 However this is not the case for Leibniz crossed modules (see \cite{FCas}). Nevertheless, it is possible to give a definition by equations:
\begin{definition}
		A \emph{representation of a Leibniz crossed module $(\lb{q},\lb{p},\eta)$} is an abelian Leibniz crossed module $(N,M,\mu)$ endowed with:
			\begin{itemize}
			\item[(i)] Actions of the Leibniz algebra $\lb{p}$ (and so $\lb{q}$ via $\eta$) on $N$ and $M$, such that the homomorphism $\mu$
			is $\lb{p}$-equivariant, that is
			\begin{align}
			\mu([p,n]) & =[p,\mu(n)], \label{p_equivariant_Lb_1} \tag{LbEQ1} \\
			\mu([n,p]) & =[\mu(n),p], \label{p_equivariant_Lb_2} \tag{LbEQ2}
			\end{align}
			for all $n\in N$ and $p\in \lb{p}$.		
			\item[(ii)] Two $K$-bilinear maps $\xi_1 \colon \lb{q}\times M\to N$ and $\xi_2 \colon M\times \lb{q}\to N$ such that
			\begin{align}
			\mu \xi_2(m,q) & =[m,q],\label{action_Lb_1a} \tag{LbM1a} \\
			\mu \xi_1(q,m) & =[q,m],\label{action_Lb_1b} \tag{LbM1b} \\
			\xi_2 (\mu(n),q) & = [n,q],\label{action_Lb_2a} \tag{LbM2a} \\
			\xi_1 \big(q,\mu(n)\big) & = [q,n],\label{action_Lb_2b} \tag{LbM2b} \\
			\xi_2 (m,[p,q]) & = \xi_2 ([m,p],q) - [\xi_2(m,q),p], \label{action_Lb_3a} \tag{LbM3a} \\
			\xi_1 ([p,q],m) & = \xi_2 ([p,m],q) - [p,\xi_2(m,q)], \label{action_Lb_3b} \tag{LbM3b} \\
			\xi_2 (m,[q,p]) & = [\xi_2(m,q),p] - \xi_2 ([m,p],q), \label{action_Lb_3c} \tag{LbM3c} \\
			\xi_1 ([q,p],m) & = [\xi_1(q,m),p] - \xi_1 (q,[m,p]), \label{action_Lb_3d} \tag{LbM3d} \\
			\xi_2 (m,[q,q']) & = [\xi_2(m,q),q'] - [\xi_2(m,q'),q], \label{action_Lb_4a} \tag{LbM4a} \\
			\xi_1 ([q,q'],m) & = [\xi_1(q,m),q'] - [q,\xi_2(m,q')], \label{action_Lb_4b} \tag{LbM4b} \\
			\xi_1 (q,[p,m]) &= - \xi_1(q,[m,p]), \label{action_Lb_5a} \tag{LbM5a} \\
			[p,\xi_1(q,m)] & = - [p,\xi_2(m,q)], \label{action_Lb_5b} \tag{LbM5b}
			\end{align}
			 for all $q,q'\in \lb{q}$, $p\in \lb{p}$, $n \in N$ $m,m'\in M$.
		\end{itemize}
\end{definition}
A morphism between two representations $(N,M,\mu)$ and $(N',M',\mu')$ of a Leibniz crossed module $(\lb{q},\lb{p},\eta)$
is a morphism of abelian Leibniz crossed modules $(f_N,f_M) \colon (N,M,\mu) \to (N',M',\mu')$ that preserves the actions together with the morphisms from $(ii)$.

\begin{remark}
	As in Remark~\ref{defrep}, given a $(\lb{q},\lb{p},\eta)$-representation $(N,M,\mu)$, we can obtain a Leibniz crossed module structure on
 $(N \rtimes  \lb{q}, M \rtimes  \lb{p}, \mu \oplus \eta)$ where $(N,M,\mu)$ is an  abelian crossed ideal and $(\lb{q},\lb{p},\eta)$ is a crossed submodule of
  $(N \rtimes  \lb{q}, M \rtimes  \lb{p}, \mu \oplus \eta)$ respectively.
  The converse statement is also true. Moreover, a representation can be seen as an action of $(\lb{q},\lb{p},\eta)$ over an abelian Leibniz crossed module $(N,M,\mu)$ in the sense of \cite{FCas}.
\end{remark}

\section{Universal enveloping crossed module of algebras of a Leibniz crossed module}\label{section_extension}
Let $(\lb{q},\lb{p},\eta)$ be a Leibniz crossed module and consider its corresponding $\cat^1$-Leibniz algebra
\[
\xymatrix{ \lb{q} \rtimes \lb{p} \ar@<0.4ex>[r]^-{s}
	\ar@<-1.2mm>[r]_-{t} & \lb{p}},
\]
with $s(q,p)=p$ and $t(q,p)=\eta(q)+p$ for all $(q,p) \in \lb{q} \rtimes \lb{p}$. Now, if we apply $\UL$ to the previous diagram, we get
\[
\xymatrixcolsep{3pc}\xymatrix{ \UL (\lb{q} \rtimes \lb{p}) \ar@<0.4ex>[r]^-{\UL(s)}
	\ar@<-1.2mm>[r]_-{\UL(t)} & \UL(\lb{p})}.
\]
Although it is true that $\UL(s)|_{\UL(\lb{p})}=\UL(t)|_{\UL(\lb{p})}=\id_{\UL(\lb{p})}$, in general, the second condition for $\cat^1$-algebras \eqref{CAs2} is not satisfied.
Nevertheless, we can consider the quotient $\ol{\UL} (\lb{q} \rtimes \lb{p})=\UL (\lb{q} \rtimes \lb{p})\slash \mathcal{X}$, where $\mathcal{X}=\Ker \UL(s) \Ker \UL(t) + \Ker \UL(t) \Ker \UL(s)$,
 and the induced morphisms $\ol{\UL}(s)$ and $\ol{\UL}(t)$. In this way, the diagram
\[
\xymatrixcolsep{3pc}\xymatrix{ \ol{\UL} (\lb{q} \rtimes \lb{p}) \ar@<0.4ex>[r]^-{\ol{\UL}(s)}
	\ar@<-1.2mm>[r]_-{\ol{\UL}(t)} & \UL(\lb{p})}
\]
is clearly a $\cat^1$-algebra. Note that $\UL(\lb{p})$ can be regarded as a subalgebra of $\ol{\UL}(\lb{q}\rtimes \lb{p})$.

We can now define $\XUL(\lb{q},\lb{p},\eta)$ as the crossed module of associative algebras given by $(\Ker\ol{\UL}(s),\UL(\lb{p}),\ol{\UL}(t)|_{\Ker\ol{\UL}(s)})$.
 This construction defines a functor $\XUL \colon \XLb \to \XAs$.

Immediately below we prove a very helpful lemma which gives us a proper description of $\Ker\UL(s)$.

\begin{lemma}\label{lemma}
The elements of the form $(q_1, p_1) \t \cdots \t (q_k, p_k)$ such that there exist $1\leq i \leq k$ with $p_i = 0$, generate $\Ker\UL(s)$.
\end{lemma}

\begin{proof}
Let $J$ be the ideal of $\T(\lb{p}^l \oplus \lb{p}^r)$ generated by the three relations of Definition~\ref{def_uea}.
Let $I$ be the ideal of $\T\big((\lb{q}\rtimes \lb{p})^l \oplus (\lb{q}\rtimes \lb{p})^r\big)$ generated by the preimage by $\T(s)$ of those relations. Then $I$ is the ideal generated by
\begin{align*}
{}& (q_1, p)_r \t (q_2, p')_r - (q_3, p')_r \t (q_4, p)_r - (q_5, [p, p'])_r, \\
{}& (q_1, p)_r \t (q_2, p')_l - (q_3, p')_l \t (q_4, p)_r - (q_5, [p, p'])_l, \\
{}& (q_1, p)_r \t (q_2, p')_l + (q_1, p)_l \t (q_2, p')_l.
\end{align*}
Additionally, the kernel of $\T(s)$ is generated by elements as those in the statement of this lemma and by elements of the form
$(q_1, p_1)_{\alpha_1} \t \cdots \t (q_k, p_k)_{\alpha_k} - (q'_1, p_1)_{\alpha_1} \t \cdots \t (q'_k, p_k)_{\alpha_k}$,
where $\alpha_k$ can be $r$ or $l$.  Since $\T(s)$ is surjective, the kernel of $\UL(s)$ will be generated by $I$ and $\Ker\T(s)$.
 Let us check that all these generators are of the claimed form.

Given $(q_1, p')_{\alpha_1} \t (q_2, p)_{\alpha_2} - (q_3, p')_{\alpha_1} \t (q_4, p)_{\alpha_2} \in \Ker\T(s)$, we have that
\begin{align*}
&(q_1, p')_{\alpha_1} \t (q_2, p)_{\alpha_2} - (q_3, p')_{\alpha_1} \t (q_4, p)_{\alpha_2}  \\
{} &= (q_1, p')_{\alpha_1} \t (q_2, p)_{\alpha_2} - (q_3, p')_{\alpha_1} \t (q_4, p)_{\alpha_2} + (q_1, p')_{\alpha_1} \t (q_4, p)_{\alpha_2} - (q_1, p')_{\alpha_1} \t (q_4, p)_{\alpha_2} \\
{} &= (q_1, p')_{\alpha_1} \t (q_2 - q_4, 0)_{\alpha_2} + (q_1 - q_3, 0)_{\alpha_1} \t (q_4, p)_{\alpha_2}.
\end{align*}
By induction one can easily derive that the elements in $\Ker\T(s)$ are of the expected form.

Let us take $(q_1, p)_r \t (q_2, p')_r - (q_3, p')_r \t (q_4, p)_r - (q_5, [p, p'])_r \in I$. Then,
\begin{align*}
(q_1, p)_r \t &(q_2, p')_r - (q_3, p')_r \t (q_4, p)_r - (q_5, [p, p'])_r \\
{} &= (q_1, p)_r \t (q_2, p')_r - (q_4, p)_r \t (q_3, p')_r + [(q_4, p), (q_3, p')]_r - (q_5, [p, p'])_r \\
{} &= (q_1, p)_r \t (q_2, p')_r - (q_4, p)_r \t (q_3, p')_r + (q', [p, p'])_r - (q_5, [p, p'])_r \\
{} &= (q_1, p)_r \t (q_2, p')_r - (q_4, p)_r \t (q_3, p')_r + (q'- q_5, 0)_r,
\end{align*}
and then we proceed as in the previous case. For the second and third identities the argument is similar.
\end{proof}
Observe that there are full embeddings
\[
\I_0, \I_1 \colon \As\lra\XAs \qquad (\text{resp.} \quad \J_0, \J_1 \colon \Lb\lra\XLb)
\]
defined, for an associative algebra $A$ (resp. for a Leibniz algebra $\lb{p}$), by $\I_0(A)=(\{0\},A,0)$, $\I_1(A)=(A,A,\id_{A})$
\big(resp. $\J_0(\lb{p})=(\{0\},\lb{p},0)$, $\J_1(\lb{p})=(\lb{p},\lb{p},\id_{\lb{p}})$\big).
The functor $\XUL\colon\XLb\to\XAs$ is a natural generalization of the functor $\UL$, in the sense that it makes the following diagram commute,
\[
\xymatrix {
	\Lb \ar[d]_{\UL}\ar[r]^{\J_0} & \XLb \ar[d]^{\XUL}  \\
	\As \ar[r]_{\I_0} & \XAs
}
\]
Regarding the embeddings $\I_1$ and $\J_1$, we have the following result.
\begin{proposition}
	There is a natural isomorphism of functors
	\[
	\XUL \circ \J_1 \cong \I_1 \circ \UL .
	\]
\end{proposition}
\begin{proof}
	Let $\lb{p}\in \Lb$. It is necessary to prove that $\XUL (\lb{p},\lb{p},\id_{\lb{p}})$ is naturally isomorphic to
 $(\UL(\lb{p}),\UL(\lb{p}),\id_{\UL(\lb{p})})$. In order to do so, we will show that $(\ol{\UL}(t)|_{\Ker\ol{\UL}(s)},\id_{\UL(\lb{p})})$
 is an isomorphism of crossed modules of algebras between $(\Ker \ol{\UL} (s), \UL (\lb{p}), \ol{\UL}(t)|_{\Ker\ol{\UL}(s)})$ and $(\UL (\lb{p}),\UL (\lb{p}),\id_{\UL(\lb{p})})$.

	It is easy to check that $(\ol{\UL}(t)|_{\Ker\ol{\UL}(s)},\id_{\UL(\lb{p})})$ is indeed a morphism of crossed modules of algebras.
	Recall that the first step in the construction of $\XUL (\lb{p},\lb{p},\id_{\lb{p}})$ requires us to consider the $\cat^{1}$-Leibniz algebra
	\[
	\xymatrix{ \lb{p} \rtimes \lb{p} \ar@<0.4ex>[r]^-{s}
		\ar@<-1.2mm>[r]_-{t} & \lb{p}},
	\]
	 with $s(p,p')=p'$ and $t(p,p')=p+p'$ for all $p,p'\in\lb{p}$. Let us define the Leibniz homomorphism $\epsilon\colon\lb{p} \to \lb{p}\rtimes\lb{p}$, $\epsilon(p)=(p,0)$.
 It is clear that $s\epsilon=0$ and $t\epsilon=\id_{\lb{p}}$.
	
	The next step is to apply the functor $\UL$ on the previous $\cat^{1}$-Leibniz algebra and take the quotient of
 $\UL(\lb{p}\rtimes\lb{p})$ by $\mathcal{X}=\Ker \UL(s) \Ker \UL(t) + \Ker \UL(t) \Ker \UL(s)$ in order to guarantee that we have a $\cat^{1}$-algebra. In the next diagram of algebras,
	\[
	\xymatrixcolsep{5pc}\xymatrix{\UL(\lb{p}) \ar[r]^{\UL(\epsilon)} & \UL(\lb{p}\rtimes\lb{p}) \ar@<0.4ex>[r]^-(.6){\UL(s)} \ar@<-1.2mm>[r]_-(.6){\UL(t)} \ar[d]_{\pi} & \UL(\lb{p}) \\
		& \ol{\UL}(\lb{p}\rtimes\lb{p})\ar@<-1.1ex>[ur]^-(.3){\ol{\UL}(s)} \ar@<-4mm>[ur]_-(.3){\ol{\UL}(t)} &
		}
	\]
	 where $\pi$ is the canonical projection,
 it is easy to see that $\ol{\UL}(s)\pi\UL(\epsilon)=\UL(s)\UL(\epsilon)=\UL(s\epsilon)=0$
 and $\ol{\UL}(t)\pi\UL(\epsilon)=\UL(t)\UL(\epsilon)=\UL(t\epsilon)=\id_{\UL(\lb{p})}$. Hence $\pi\UL(\epsilon)$ takes values in $\Ker\ol{\UL}(s)$
and it is a right inverse for $\ol{\UL}(t)|_{\Ker\ol{\UL}(s)}$.
	
	Now we need to show that $\pi\UL(\epsilon)\ol{\UL}(t)=\id_{\Ker\ol{\UL}(s)}$. Note that $\mathcal{X}\subset\Ker\UL(s)$, so $\Ker \ol\UL(s)= \Ker\UL(s)/\mathcal{X}$ and,
 as proved in Lemma~\ref{lemma}, $\Ker \UL(s)$ is generated by all the elements of the form
	\begin{equation}\label{ker_ULs_element}
	(p_{1},p'_{1})\otimes \dots \otimes (p_i,0)\otimes\dots\otimes (p_k,p'_k)
	\end{equation}
	with $p_j,p'_j\in \lb{p}$, $1 \leq i,j\leq k$. By the definition of $\UL(t)$ and
	$\UL(\epsilon)$, the value of $\UL(\epsilon)\UL(t)$ on \eqref{ker_ULs_element} is
	\begin{equation}\label{ker_ULs_element_image}
	(p_{1}+p'_{1},0)\otimes \dots \otimes (p_i,0)\otimes\dots\otimes (p_k+p'_k,0).
	\end{equation}
	Furthermore, one can easily derive that, in $\Ker\UL(s)/\mathcal{X}$,
	\begin{align*}
	&(p_{1}+p'_{1},0)\otimes \dots \otimes (p_i,0)\otimes\dots\otimes (p_k+p'_k,0)\\
	&=(p_{1},p'_{1})\otimes \dots \otimes (p_i,0)\otimes\dots\otimes (p_k+p'_k,0),
	\end{align*}
	By applying the same procedure as many times as required, one can deduce that
	\begin{align*}
	&(p_{1}+p'_{1},0)\otimes \dots \otimes (p_i,0)\otimes\dots\otimes (p_k+p'_k,0)\\
	&= (p_{1},p'_{1})\otimes \dots \otimes (p_i,0)\otimes\dots\otimes (p_k,p'_k).
	\end{align*}
	Thus, the elements \eqref{ker_ULs_element} and \eqref{ker_ULs_element_image} are equal in $\Ker\UL(s)/\mathcal{X}$ and it follows that
  \[
	\pi\UL(\epsilon)\ol\UL(t)|_{\Ker \ol\UL(s)}=\id_{\Ker \ol\UL(s)}.
	\]
	Therefore we have found an inverse for the morphism of crossed modules of algebras $(\ol{\UL}(t)|_{\Ker\ol{\UL}(s)},\id_{\UL(\lb{p})})$. It is fairly easy to prove that this construction is natural.
\end{proof}
\section{Isomorphism between the categories of representations}\label{section_isomorph}
In this section, we give the construction of an isomorphism between the categories of representations of a Leibniz crossed module
and left modules over its corresponding universal enveloping crossed module of algebras. Recall that the method used in the proof of the equivalent result
 in the case of Lie algebras cannot be applied in our case due to the lack of actor in the category of Leibniz crossed modules.
\begin{theorem}
	The category of representations of a Leibniz crossed module $(\lb{q},\lb{p},\eta)$ is isomorphic to the category of left modules over its
universal enveloping crossed module of algebras $\XUL(\lb{q},\lb{p},\eta)$.
\end{theorem}

\begin{proof}
Let $(N, M, \mu)$ be a left $(\Ker\ol{\UL}(s),\UL(\lb{p}),\ol{\UL}(t)|_{\Ker\ol{\UL}(s)})$-module. Then we have a homomorphism
\[
(\vp,\psi) \colon (\Ker\ol{\UL}(s),\UL(\lb{p}),\ol{\UL}(t)|_{\Ker\ol{\UL}(s)}) \to (\Hom_K(M,N), \End(N,M,\mu), \Gamma),
\]
such that $\psi \circ \ol{\UL}(t)|_{\Ker\ol{\UL}(s)} = \Gamma \circ \vp$, $\vp(ba) = \vp(b)\psi(a)$ and $\vp(ab) = \psi(a)\vp(b)$.
 We need to define actions of $\lb{p}$ on $N$ and $M$ satisfying  \eqref{p_equivariant_Lb_1}  and \eqref{p_equivariant_Lb_2}
  and we need to define $\xi_1 \colon \lb{q}\times M\to N$ and $\xi_2 \colon M\times \lb{q}\to N$ satisfying identities \eqref{action_Lb_1a}--\eqref{action_Lb_5b}.

We define the actions of $\lb{p}$ on $N$ and $M$ as those induced by $\psi \colon \UL(\lb{p}) \to \End(N,M,\mu)$ as in Theorem~\ref{theo_equiv_LoPi}.
 The identities \eqref{p_equivariant_Lb_1} and \eqref{p_equivariant_Lb_2} are followed by the properties of $\End(N,M,\mu)$.
  We define the morphisms $\xi_1$ and $\xi_2$ by $\xi_1(q, m) = \vp\big((q,0)_l\big)(m)$ and $\xi_2(m, q) = \vp\big((q,0)_r\big)(m)$.
   The identities \eqref{action_Lb_1a}, \eqref{action_Lb_2a} and \eqref{action_Lb_1b}, \eqref{action_Lb_2b}  are followed by the commutative square
   $\psi \circ \ol{\UL}(t)|_{\Ker\ol{\UL}(s)} = \Gamma \circ \vp$ applied to the elements $(q, 0)_r$ and $(q, 0)_l$ respectively.
    Given the element $([p, q], 0)_r \in \Ker\ol{\UL}(s)$, we have that
\[
([p, q], 0)_r = [(0, p)_r, (q, 0)_r] = (0,p)_r(q,0)_r - (q, 0)_r(0,p)_r.
\]
Applying $\vp$ to this relation and using that $\vp(ba) = \vp(b)\psi(a)$ and $\vp(ab) = \psi(a)\vp(b)$ we obtain
$\vp\big(([p, q], 0)_r\big) = \psi_2(p_r)\vp\big((q, 0)_r\big) - \vp\big((q, 0)_r\big)\psi_1(p_r)$ which implies \eqref{action_Lb_3a}.
 Proceeding in the same way for the elements $([p, q], 0)_l$, $([q, p], 0)_r$ and $([q, p], 0)_l$ we check that identities \eqref{action_Lb_3b}, \eqref{action_Lb_3c} and \eqref{action_Lb_3d} are satisfied.
  Doing a similar argument on the elements $([q, q'], 0)_r$ and $([q, q'], 0)_l$ we obtain identities \eqref{action_Lb_4a} and \eqref{action_Lb_4b}. Applying $\vp$ to the relations
\[
(0, p)_l (q, 0)_l = -(0, p)_r(q, 0)_l \qquad \text{and} \qquad (q, 0)_l (0, p)_l = -(q, 0)_r(0, p)_l
\]
we have identities \eqref{action_Lb_5a} and \eqref{action_Lb_5b} respectively.

Conversely, let $(N, M, \mu)$ be a  $(\lb{q},\lb{p},\eta)$-representation. We need to construct a morphism of crossed modules of algebras
 $(\vp,\psi)$ from $\XUL(\lb{q},\lb{p},\eta)$ to $(\Hom_K(M,N), \End(N,M,\mu), \Gamma)$.
 The homomorphism $\psi = (\psi_1, \psi_2) \colon \UL(\lb{p}) \to \End(N, M, \mu)$ is the homomorphism induced by the actions of
 $\lb{p}$ on $N$ and $M$ as in Theorem~\ref{theo_equiv_LoPi}. It is well defined due to identities \eqref{p_equivariant_Lb_1} and \eqref{p_equivariant_Lb_2}.
  Consider the homomorphism $\Phi \colon (\lb{q} \rtimes \lb{p})^l \oplus (\lb{q} \rtimes \lb{p})^r \to \Hom_K(N \oplus M, N \oplus M)$ defined by
\begin{align*}
\Phi(q, p)_l(n, m) & = ([q, n] + [p, n] + \xi_1(q, m), [p, m]), \\
\Phi(q, p)_r(n, m)  & = ([n, q] + [n, p] + \xi_2(m, q), [m, p]).
\end{align*}
Note that they can also be rewritten as
\begin{align*}
\Phi(q, p)_l(n, m) & = \big(t(q, p)_l(n) + \xi_1(q, m), s(q, p)_l(m)\big),\\
\Phi(q, p)_r(n, m)  & = \big(t(q, p)_r(n) + \xi_2(m, q),  s(q, p)_r(m)\big).
\end{align*}
By the universal property of the tensor algebra, there is a unique homomorphism
\[
\T(\Phi)\colon \T\big((\lb{q} \rtimes \lb{p})^l \oplus (\lb{q} \rtimes \lb{p})^r\big) \to \Hom_K(N \oplus M, N \oplus M),
\]
commuting with the inclusion.

We consider the projection $\pi \colon \Hom_K(N \oplus M, N \oplus M) \to \Hom_K(M, N \oplus M)$, where $\pi(f)(m) = f(0, m)$, and denote $\vp' = \pi \circ \T(\Phi)$.
 Given an element of the form $(q',p')_r(q, p)_r - (q, p)_r(q',p')_r + [(q, p)_r, (q',p')_r]$ we obtain that

\begin{align*}
\vp'\big((q',p')_r(q, p)_r - (q, p)_r&(q',p')_r + [(q, p)_r, (q',p')_r]\big)(m) \\
{} &= [\xi_2(m, q'), q] + \big[[m, p'], q\big] + [\xi_2(m, q'), p] \\
{} &- [\xi_2(m, q), q'] - \big[[m, p], q'\big] - [\xi_2(m, q), p'] \\
{} &+ \xi_2(m, [q, q']) + \xi_2(m, [q, p']) + \xi_2(m, [q', p]) = 0,
\end{align*}
by the properties of $\xi_2$.

Analogously, it is possible to prove that $\vp'$ vanishes on the other two relations of the universal enveloping algebra.
 Then $\vp'$ factors through $\UL(\lb{q}\rtimes\lb{p})$. In order to ease notation we will refer to it as $\vp'$ as well.

By definition it is clear that $\vp'|_{\Ker \UL(s)}(M)\subseteq N$ and
 $\T(\Phi)\big((\lb{q} \rtimes \lb{p})^l \oplus (\lb{q} \rtimes \lb{p})^r\big)(N) = \T(t)\big((\lb{q} \rtimes \lb{p})^l \oplus (\lb{q} \rtimes \lb{p})^r\big)(N)$.
  Then $\vp'$ factors through $\Ker\UL(t)\Ker\UL(s)$. Moreover, we have that
\begin{align*}
\vp'\big((q&, p)_r(q', p')_r\big)(m) = \big( [\xi_2(m, q), q'] + [\xi_2(m, q), p'] + \xi_2([m, p], q), \big[[m, p], p'\big] \big) \\
{} &= \big( \xi_2(m, [t(q, p)_r, q']) + [\xi_2(m, q'), t(q, p)_r] + [\xi_2(m, q), s(q', p')_r], s(q', p')_rs(q, p)_rm \big).
\end{align*}
Extending this argument we see that $\vp'$ also factors through $\Ker\UL(s)\Ker\UL(t)$.
\[
\xymatrix{
(\lb{q} \rtimes \lb{p})^l \oplus (\lb{q} \rtimes \lb{p})^r \ar[r]^-{\Phi} \ar[d] & \Hom_K(N \oplus M, N \oplus M) \ar[r]^-{\pi} & \Hom_K(M, N \oplus M) \\
\T\big((\lb{q} \rtimes \lb{p})^l \oplus (\lb{q} \rtimes \lb{p})^r\big) \ar[ur]_-{\T(\Phi)} \ar[d] \\
\ol\UL(\lb{q} \rtimes \lb{p}) \ar@/_3pc/[rruu]^-{\vp'}
}
\]
Therefore, $\vp$ will be the restriction of $\vp'$ to $\Ker \ol\UL(\lb{q} \rtimes \lb{p})$ and it will take values in $\Hom_K(M, N)$, that is
\[
\vp \colon \Ker \ol\UL(\lb{q} \rtimes \lb{p}) \to \Hom_K(M, N).
\]
With these definitions of $\vp$ and $\psi$, to check that
\[
(\vp,\psi) \colon (\Ker\ol{\UL}(s),\UL(\lb{p}),\ol{\UL}(t)|_{\Ker\ol{\UL}(s)}) \to (\Hom_K(M,N), \End(N,M,\mu), \Gamma)
\]
is a morphism of crossed modules of algebras is now a matter of straightforward computations.
\end{proof}

\section{Relation with the Loday-Pirashvili category}\label{section_LM}

In \cite{LoPi2}, Loday and Pirashvili introduced a very interesting way to see Leibniz algebras as Lie algebras over another tensor category different from $K$-Mod,
 the tensor category of linear maps, denoted by $\LM$. In this section we will extend this construction to crossed modules
  and check that the relation with the universal enveloping algebra still holds in the $2$-dimensional case.

\begin{definition}[\cite{LoPi2}]
	Let $M$ and $\g$ be $K$-modules. The objects in $\LM$ are $K$-module homomorphisms $(M \overset{\alpha}\to \g)$.
 In order to ease notation we will simply write $(M, \g)$ if there is no possible confusion.
	Given two objects $M \overset{\alpha}\to \g$ and $N \overset{\beta} \to \h$, an arrow is a pair of $K$-module homomorphisms
$\varrho_1\colon M \to N$ and $\varrho_2\colon \g \to \h$ such that $\beta\circ\varrho_1 = \varrho_2 \circ \alpha$.
	$\LM$ is a tensor category with the tensor product defined as
	\[
	(M \overset{\alpha}\to \g) \t (N \overset{\beta}\to \h) = \big((M \t \h) \oplus (\g \t N) \xrightarrow{\alpha \otimes 1_{\h} + 1_{\g} \otimes \beta} \g \t \h\big).
	\]
	
	An \emph{associative algebra} in $\LM$ is an object $(A \overset{\beta}\to R)$ where $R$ is an associative $K$-algebra, $A$ is a $R$-bimodule and $\beta$ is a homomorphism of $R$-bimodules.
	
	A \emph{Lie algebra} in $\LM$ is an object $(M \overset{\alpha}\to \g)$ where $\g$ is a Lie algebra, $M$ is a right $\g$-representation and $\alpha$ is $\g$-equivariant.
 Given a Lie algebra object in $\LM$, its \emph{universal enveloping algebra} in $\LM$ is $(\U(\g) \t M \to \U(\g))$, $1 \t m \mapsto \alpha(m)$. The action is given by
	\[
	g(x \t m) = gx \t m \qquad \text{and} \qquad (x \t m)g = xg \t m + x \t [m, g], \ g,x \in \g,  \, m \in M.
	\]
\end{definition}

A Leibniz algebra  $\p$  can be viewed as a Lie algebra object in $\LM$, namely $ \p \to \Liez(\p)$.

\begin{definition}
	Let $(A \overset{\alpha}\to R)$ and $(B \overset{\beta}\to S)$ be two associative algebras in $\LM$. We say there is an \emph{action} of $(A, R)$ on $(B, S)$ if we have the following:
	\begin{itemize}
		\item An action of algebras of $R$ on $S$;
		\item an $R$-bimodule structure on $B$, compatible with the action of $S$;
		\item two homomorphisms $\xi_1 \colon A \t_R S \to B$ and $\xi_2 \colon S \t_R A \to B$, such that $\xi_1(a, s)s' = \xi_1(a, ss')$, $s\xi_1(a, s') = \xi_2(s, a)s'$, $s\xi_2(s', a) = \xi_2(ss', a)$;
		\item $\beta$ is also a homomorphism of $R$-bimodules, such that $\beta\big(\xi_1(a, s)\big) = \alpha(a)s$ and $\beta\big(\xi_2(s, a)\big) = s\alpha(a)$.
	\end{itemize}
	A \emph{crossed module of associative algebras} in $\LM$ is an arrow $(\omega_1, \omega_2) \colon (B, S) \to (A, R)$ and an action of $(A, R)$ on $(B, S)$ such that
	\begin{itemize}
		\item $\omega_2$ with the action of $R$ on $S$ is a crossed module of associative algebras;
		\item $\omega_1$ is a homomorphism of $R$-bimodules satisfying
$a \omega_2(s) = \omega_1\big(\xi_1(a, s)\big)$, $\omega_2(s)a = \omega_1\big(\xi_2(s, a)\big)$, $\xi_1\big(\omega_1(b), s\big) = b \omega_2(s) = bs$ and $\xi_2\big(s, \omega_1(b)\big) = \omega_2(s)b = sb$.
	\end{itemize}
\end{definition}

Let $(\omega_1, \omega_2) \colon (B, S,\beta) \to (A, R,\alpha)$ be a crossed module of algebras in $\LM$. We can associate to it the crossed module of algebras
\begin{equation}\label{associated}
(B \oplus S, A \oplus R, \omega_1 \oplus \omega_2),
\end{equation}
where
\begin{align*}
  (a,r)(a',r') & = (\alpha(a)a' + ar'+ra',rr'), &&  a,a' \in A, \ r,r' \in R, \\
  (b,s)(b',s') & = (\beta(b)b' + bs'+sb',ss'), &&  b,b' \in B, \ s,s' \in S.
\end{align*}

\begin{definition}
	Let $(M \overset{\alpha}\to \g)$ and $(N \overset{\beta}\to \h)$ be two Lie algebras in $\LM$. We say there is a \emph{right action} of $(M, \g)$ on $(N, \h)$ if we have the following:
	\begin{itemize}
		\item Compatible right Lie actions of $\g$ on $\h$ and $N$;
		\item a homomorphism $\xi \colon M \t \h \to N$ such that $[\xi(m, h), g] = \xi([m, g], h) + \xi(m, [h, g])$ and $\xi(m, [h, h']) = [\xi(m, h), h'] - [\xi(m, h'), h]$;
		\item $\alpha$ is $\h$-equivariant and it satisfies that $\alpha\big(\xi(m, h)\big) = [\alpha(m), h]$.
	\end{itemize}
	A \emph{crossed module of Lie algebras} in $\LM$ is an arrow $(\varrho_1, \varrho_2) \colon (N, \h) \to (M, \g)$ and an action of $(M, \g)$ on $(N, \h)$ such that
	\begin{itemize}
		\item $\varrho_2$ with the right Lie action of $\g$ on $\h$ is a crossed module of Lie algebras;
		\item $\varrho_1$ is a $\g$-equivariant homomorphism such that $\varrho_1\big(\xi(m, h)\big) = [m, \varrho_2(h)]$ and $[n, h] = \xi(\varrho_1(n), h) = [n, \varrho_2(h)]$.
	\end{itemize}
\end{definition}

Let $(\varrho_1, \varrho_2) \colon (N, \h) \to (M, \g)$ be a Lie crossed module in $\LM$.
 We construct the semidirect product in $\LM$ by obtaining a Lie object $(N \oplus M, \h \rtimes \g)$,
 where $[(n, m), (h, g)] = ([n, h] + [n, g] + \xi(m, h), [m, g])$. Let $(s_1, s_2)$ and $(t_1, t_2)$ be two arrows in $\LM$
\[
\xymatrix{
	N \oplus M \ar[d]_{\beta \oplus \alpha} \ar@<0.4ex>[r]^-{s_1} \ar@<-1.2mm>[r]_-{t_1} & M \ar[d]^{\alpha} \\
	\h \rtimes \g \ar@<0.4ex>[r]^-{s_2} \ar@<-1.2mm>[r]_-{t_2} & \g
	}
\]
where $s_1(n, m) = m$, $s_2(h, g) = g$ and $t_1(n, m) = \varrho_1(n) + m$, $t_2(h, g) = \varrho_2(h) + g$.
We apply the universal enveloping algebra functor in $\LM$ to the previous diagram.
\[
\xymatrix{
	\U(\h \rtimes \g) \t (N \oplus M) \ar[d]_{\U(\beta \oplus \alpha)} \ar@<0.4ex>[r]^-{\U(s_1)} \ar@<-1.2mm>[r]_-{\U(t_1)} & \U(\g) \t M \ar[d]^{\U(\alpha)} \\
	\U(\h \rtimes \g) \ar@<0.4ex>[r]^-{\U(s_2)} \ar@<-1.2mm>[r]_-{\U(t_2)} & \U(\g)
}
\]

Considering $\U(\g)$ as a subalgebra of $\U(\h\rtimes \g)$, there are induced algebra actions of $\U(\g)$ on $\U(\h \rtimes \g)$ and on $\U(\h \rtimes \g) \t (N \oplus M)$.
 There are also two morphisms
\begin{align*}
\xi_1 & \colon (\U(\g) \t M) \t_{\U(\g)} \U(\h \rtimes \g) \to \U(\h \rtimes \g) \t (N \oplus M),\\[.1cm]
\xi_2 & \colon \U(\h \rtimes \g) \t_{\U(\g)} (\U(\g) \t M) \to \U(\h \rtimes \g) \t (N \oplus M),
\end{align*}
where $\xi_1\big( (1 \t m), (h, g) \big) = (h, g) \t (0, m) + 1 \t (\xi(m, h), [m, g])$ and $\xi_2\big((h, g), (1 \t m)\big) = (h, g) \t (0, m)$,
 where $\xi \colon M \t \h \to N$ is the homomorphism from the action of $(M, \g)$ on $(N, \h)$. These actions and homomorphisms define an action of algebras in $\LM$.

Let us consider the ideal of $\U(\h \rtimes \g) \t (N \oplus M)$ given by
$\mathcal{Y}' = \Ker \U(s_1) \Ker \U(t_2) + \Ker \U(s_2) \Ker \U(t_1) + \Ker \U(t_1) \Ker \U(s_2) + \Ker \U(t_2) \Ker \U(s_1)$ and the ideal of $\U(\h \rtimes \g)$ given by
$\mathcal{X}' =  \Ker \U(s_2) \Ker \U(t_2) + \Ker \U(t_2) \Ker \U(s_2)$.

\begin{lemma}\label{actcm}
	The two squares
	\[
	\xymatrix{
		\dfrac{\U(\h \rtimes \g) \t (N \oplus M)}{\mathcal{Y}'} \ar[d]_{\ol\U(\beta \oplus \alpha)} \ar@<0.4ex>[r]^-{\ol\U(s_1)} \ar@<-1.2mm>[r]_-{\ol\U(t_1)} & \U(\g) \t M \ar[d]^{\U(\alpha)} \\
		\dfrac{\U(\h \rtimes \g)}{\mathcal{X}'} \ar@<0.4ex>[r]^-{\ol\U(s_2)} \ar@<-1.2mm>[r]_-{\ol\U(t_2)} & \U(\g)
	}
	\]
	are well defined. Moreover, the actions of $\U(\g)$ on $\U(\h \rtimes \g)$ and $\U(\h \rtimes \g) \t (N \oplus M)$ and the morphisms $\xi_1$ and $\xi_2$ factor through $\mathcal{X}'$ and $\mathcal{Y}'$.
\end{lemma}

\begin{proof}
	Since the bottom row is the Lie algebra case, the proof can be found in \cite{FCas}. The top row follows by definition of $\mathcal{Y}'$.
 The homomorphism $\ol\U(\beta \oplus \alpha)$ is zero in $\mathcal{Y}'$ by the equivariance of $\alpha \oplus \beta$ and the commutativity of the diagram.
  Again, the action of $\U(\g)$ on $\mathcal{X}'$ is zero since we are exactly in the Lie case.
   It is obvious that $\mathcal{X}'$ acts trivially on $\U(\h \rtimes \g) \t (N \oplus M)$ and that $\U(\h \rtimes \g)$, and consequently $\U(\g)$, acts trivially on $\mathcal{Y}'$.
    Let $1 \t m \in \U(\g) \t M$. Then $\xi_2\big(\Ker \U(s_2) \Ker \U(t_2), 1 \t m\big) = \Ker \U(s_2) \Ker \U(t_2) \t (0, m) \subseteq \mathcal{Y}'$ and the same happens for $\Ker \U(t_2) \Ker \U(s_2)$.
	On the other hand, $\xi_1\big(1 \t m, \Ker \U(s_2) \Ker \U(t_2)\big)$ is equal to
 $\Ker \U(s_2) \Ker \U(t_2) \t (0, m) \subseteq \mathcal{Y}'$ plus $\Ker \U(s_2) \Ker \U(t_2)$ acting on $1 \t (0, m)$, which clearly is also inside $\mathcal{Y}'$.
\end{proof}

We consider now $\big(\Ker \ol\U(s_1), \Ker \ol\U(s_2)\big)$.
 The restriction of $\big(\ol\U(t_1), \ol\U(t_2)\big)$ and $\ol\U(\beta \oplus \alpha)$ to this kernel will be denoted in the same way by abuse of notation. Then we have the following result.

\begin{theorem}
Let $(\varrho_1, \varrho_2) \colon (N, \h) \to (M, \g)$ be a Lie crossed module in $\LM$. The following square illustrates a crossed module of algebras in $\LM$ with the induced actions
\[
	\xymatrix{
		\Ker\ol\U(s_1) \ar[d]_{\ol\U(\beta \oplus \alpha)} \ar[r]^-{\ol\U(t_1)} & \U(\g) \t M \ar[d]^{\U(\alpha)} \\
		\Ker\ol\U(s_2) \ar[r]^-{\ol\U(t_2)} & \U(\g)
	}
\]
Moreover, it is the universal enveloping crossed module of algebras of $(\varrho_1, \varrho_2) \colon (N, \h) \to (M, \g)$ in $\LM$.
\end{theorem}

\begin{proof}
	The action is studied in Lemma~\ref{actcm}, and the bottom row is a crossed module of algebras since is just the Lie algebra case.
 The second condition is followed straightforward from definitions of $\xi_1$ and $\xi_2$.
\end{proof}

Let us now consider a crossed module of Leibniz algebras $(\q, \p, \eta)$. It can be viewed as a crossed module of Lie algebras in $\LM$ as a square
\[
\xymatrix{
	\q \ar[r]^\eta \ar[d]  &  \p \ar[d] \\
	\dfrac{\Liez(\q)}{[\q,\p]_\texttt{x}} \ar[r]^{\overline{\eta}}         & \Liez(\p)
	}
\]
We can consider its universal enveloping algebra, obtaining that way a crossed module of algebras in $\LM$

\[
\xymatrix{
	\Ker\ol\U(s_1) \ar[r]^-{\ol\U(t_1)} \ar[d] & \U\big(\Liez(\p)\big) \t \p \ar[d] \\
	\Ker\ol\U(s_2) \ar[r]^-{\ol\U(t_2)}        & \U\big(\Liez(\p)\big)
}
\]
Following the construction \eqref{associated} we obtain its corresponding crossed module of associative algebras in the classical setting
\[
\bigg(\Ker\ol\U(s_1) \oplus \Ker\ol\U(s_2), \big(\U(\Liez(\p)) \t \p\big) \oplus \U(\Liez(\p)), \big(\ol\U(t_1), \ol\U(t_2)\big)\bigg).
\]

\begin{theorem}
	Let $(\q, \p, \eta)$ be a crossed module of Leibniz algebras. Its universal enveloping crossed module of algebras $\XUL(\lb{q},\lb{p},\eta)$  is isomorphic to the crossed module of algebras defined above.
\end{theorem}

\begin{proof}
We have the following diagram
\[
\xymatrix{
	\U\big(\Liez(\q \rtimes \p)\big) \t (\q \rtimes \p) \ar[d] \ar[r]^-{\ol\U(t_1)} & \U\big(\Liez(\p)\big) \t \p \ar[d] \\
	\U\big(\Liez(\q \rtimes \p)\big) \ar[r]^-{\ol\U(t_2)}       & \U\big(\Liez(\p)\big)
}
\]
By \cite[(2.4) Proposition]{LoPi} we know that the direct sum of the two objects of the first column is isomorphic to $\UL(\q \rtimes \p)$
 and the direct sum of the objects of the second column is isomorphic to $\UL(\p)$.
  Let $\mathcal{X}$ be the ideal $\Ker \UL(s) \Ker \UL(t) + \Ker \UL(t) \Ker \UL(s)$ defined in Section~\ref{section_extension}. Consider the isomorphism
\[
\theta \colon \UL(\q \rtimes \p) \xrightarrow{\sim} \bigg(\U\Big(\dfrac{\Liez(\q)}{[\q,\p]_\texttt{x}} \rtimes \Liez(\p)\Big) \t (\q \rtimes \p)\bigg) \oplus \U\Big(\dfrac{\Liez(\q)}{[\q,\p]_\texttt{x}} \rtimes \Liez(\p)\Big)
\]
defined on generators by $\theta\big((q, p)_r\big) = (\overline{q}, \overline{p})$ and $\theta\big((q, p)_l\big) = 1 \t (\overline{q}, \overline{p})$.
 We need to check that it maps $\mathcal{X}$ to $\mathcal{X}' + \mathcal{Y}'$ and $\theta^{-1}$ maps $\mathcal{X}' + \mathcal{Y}'$ to $\mathcal{X}$, but this follows straightforward by the definitions of $\mathcal{X}$, $\mathcal{X}'$ and $\mathcal{Y}'$ completing the proof.
\end{proof}


\end{document}